\documentclass[11pt]{amsart}
\usepackage{amsmath}
\usepackage{amscd}
\usepackage{amssymb}
\usepackage{amsthm}
\usepackage{color} 
\usepackage[matrix, arrow,curve]{xy}
\usepackage{verbatim}
\usepackage{MnSymbol}

\newtheorem{lemma}{Lemma}[section]
\newtheorem{prop}[lemma]{Proposition}
\newtheorem{thm}[lemma]{Theorem}
\newtheorem{metathm}[lemma]{Metatheorem}
\newtheorem{cor}[lemma]{Corollary}

\newtheorem{conj}[lemma]{Conjecture}

\theoremstyle{definition}
\newtheorem{defn}[lemma]{Definition}

\newtheorem{ex}[lemma]{Example}

\newtheorem{rem}[lemma]{Remark}

\newcommand{\Vect}{\mathrm{-Vect}}
\newcommand{\End}{\mathrm{End}}
\newcommand{\Hom}{\mathrm{Hom}}
\newcommand{\Ext}{\mathrm{Ext}}
\newcommand{\Spec}{\mathrm{Spec}}
\newcommand{\MMN}{\mathcal{MM}_{\mathrm{Nori}}}
\newcommand{\MM}{\mathcal{MM}}
\newcommand{\mot}{\mathrm{mot}}
\newcommand{\Nori}{\mathrm{Nori}}

\newcommand{\AM}{\mathrm{AM}}

\newcommand{\Ah}{\mathcal{A}}
\newcommand{\isom}{\cong}
\newcommand{\id}{\mathrm{id}}

\newcommand{\tensor}{\otimes}

\newcommand{\im}{\mathrm{Im}}

\newcommand{\Gal}{\mathrm{Gal}}

\newcommand{\eff}{\mathrm{eff}}

\newcommand{\dR}{\mathrm{dR}}
\newcommand{\an}{\mathrm{an}}

\newcommand{\DM}{\mathsf{DM}}

\newcommand{\gm}{\mathrm{gm}}
\newcommand{\DMgm}{\DM_{\gm}}

\newcommand{\DMT}{\mathsf{DMT}}
\newcommand{\MT}{\mathcal{MT}}
\newcommand{\MTZ}{\MT^f}
\newcommand{\onemot}{1\mathrm{-Mot}}

\newcommand{\MHS}{\mathrm{MHS}}
\newcommand{\VV}{(\K,\Q)\Vect}
\newcommand{\VValg}{(\Qbar,\Q)\Vect}
\newcommand{\VVQ}{(\Q,\Q)\Vect}

\newcommand{\sing}{\mathrm{sing}}
\newcommand{\Per}{\mathcal{P}} 
\newcommand{\Perf}{\tilde{\Per}} 

\newcommand{\Mod}{\mathrm{-Mod}}
\newcommand{\trdeg}{\mathrm{trdeg}}

\newcommand{\CDTarg}[2]{C(#1,#2)}

\newcommand{\Q}{\mathbb{Q}}
\newcommand{\Qbar}{{\overline{\Q}}}
\newcommand{\QR}{\tilde{\Q}}
\newcommand{\Z}{\mathbb{Z}}
\newcommand{\R}{\mathbb{R}}
\newcommand{\C}{\mathbb{C}}
\newcommand{\G}{\mathbb{G}}
\newcommand{\K}{\mathbb{K}}
\newcommand{\A}{\mathbb{A}}
\newcommand{\Na}{\mathbb{N}}
\newcommand{\Oh}{\mathcal{O}}

\newcommand{\Gm}{\G_m}

\newcommand{\Ch}{\mathcal{C}}
\newcommand{\coker}{\mathrm{coker}}

\begin{document}
\title{Galois theory of periods}
\author{Annette Huber}
\address{Math. Institut\\ Ernst-Zermelo-Str. 1\\ 79104 Freiburg\\ Germany}
\email{annette.huber@math.uni-freiburg.de}

\date{\today}
\maketitle
\tableofcontents
\section{Introduction}
We are interested in the period algebra $\Per$, a countable subalgebra of $\C$. Its elements are roughly of the form
\[ \int_\sigma\omega\]
where $\omega$ is a closed differential form over $\Q$ or $\Qbar$ and
$\sigma$ a suitable domain of integration also defined over $\Q$ or $\Qbar$.
All algebraic numbers are periods in this sense. Moreover, many interesting transcendental number like $2\pi i$, $\log(2)$, or $\zeta(3)$ are also period numbers.  

In this mostly expository note, we want to explain how the work of Nori extends Galois theory from the field extension $\Qbar/\Q$ to $\Per/\Q$, at least conjecturally. This realises and generalises a vision of Grothendieck.

As explained in more detail below, we should formulate Galois theory as saying that the natural operation
\[ \Gal(\Qbar/\Q)\times \Spec(\Qbar)\to \Spec(\Qbar)\]
is a torsor, see Section~\ref{ssec:galois} for more details. 
By the period conjecture (in Kontsevich's formulation in \cite{kontsevich}), we also have a torsor structure
\[ G_\mot(\Q)\times \Spec(\Per)\to\Spec(\Per)\]
where $G_\mot(\Q)$ is Nori's motivic Galois group of $\Q$. This is a pro-algebraic group attached to the rigid tensor category of mixed motives over $\Q$. 
By Tannaka theory, rigid tensor subcategories of the category of motives correspond to quotients of $G_\mot(\Q)$.  The case of Artin motives  gives back ordinary Galois theory. In particular, there is a natural exact sequence
\[ 0\to G_\mot(\Qbar)\to G_\mot(\Q)\to\Gal(\Qbar/\Q)\to 0.\]
It corresponds to the inclusion $\Qbar\subset\Per$.

It is a nice feature of Nori's theory (and a key advantage over the older approach) that it works for any abelian category of motives, even if it is not closed under tensor product. We get a weaker structure that we call semi-torsor, see Definition~\ref{defn:semi}. Hence, we also have a (conjectural) generalisation of Galois theory to 
sub-vector-spaces of $\Per$. It is unconditonal in the case of $1$-motives.

While not bringing us closer to a proof of the period conjecture in general, the theory has structural consequences: for example the period conjectures for intermediate fields $\Q\subset\K\subset \Qbar$ are equivalent. Another direct consequence
is the fully faithfulness of the functor attaching to a Nori motive $M$ the pair
$(H_\dR(M),H_\sing(M),\phi)$ of its realisations linked by the period isomorphism. 
This observation connects the period conjecture to fullness conjectures in the spirit of the Hodge or Tate conjecture. One such was recently formulated by
Andreatta, Barbieri-Viale and Bertapelle, see \cite{ABB}. As explained below, their conjecture has overlap with the period conjecture, but there is no implication in either direction.

Actual evidence for the period conjecture is weak. The $0$-dimensional case
amounts to Galois theory. The $1$-dimensional case has now been completed by the author and W\"ustholz in \cite{huber-wuestholz}, generalising work by Lindemann, Gelfond, Schneider, Siegel, Baker and W\"ustholz.

The idea of the present paper was conceived after my talk at the 2017 workshop
''Galois Theory of Periods and Applications'' at MSRI. It does not aim at proving new results, but rather at making key ideas of the long texts \cite{period-buch} and \cite{huber-wuestholz} more accessible. We also take the opportunity to
discuss the relation between the various period conjectures in the literature.

\subsection*{Structure of the paper} We start by giving several, equivalent definitions of the period algebra. We then turn to the period conjecture. We explain Kontsevich's formal period algebra $\Perf$ in terms of explicit generators and relations. The period conjecture simply asserts that the obvious surjective map $\Perf\to\Per$ is also injective.
We then give a reinterpretation of $\Perf$ as a torsor. Thus the main aim of the note is achieved: a (conjectural) Galois theory of periods.

The following sections make the connections to Grothendieck's version of the period conjecture and the more recent cycle theoretic approach of
Bost/Charles and Andreatta/Barbieri-Viale/Bertapelle.

The last section inspects some examples of special period spaces studied in the literature: Artin motives ($\Qbar$), mixed Tate motives (multiple zeta-values) and $1$-motives (periods of curves).

\subsection*{Acknowledgements}I thank my coauthors Stefan M\"uller-Stach and
Gisbert W\"ustholz for many insights and stimulating discussions. My ideas on the period conjecture were also heavily influenced by a discussion with Yves Andr\'e during a visit to Freiburg. The aspect of a Galois theory of periods
is also stressed in his survey articles \cite{andre1}, \cite{andre3}.

\section{Period spaces}
When we talk about periods, we mean a certain countable subalgebra  $\Per$ of
$\C$. It has several, surprisingly different, descriptions. Roughly
a period is a number of the form
\[ \int_\sigma\omega\]
where both $\omega$ and $\sigma$ are something defined over $\Q$ or
$\Qbar$. 

There are many possible ways to specify what $\omega$ and $\sigma$ to use. We will go through a list of possible choices:
nc-periods (or just periods) in Section~\ref{ssec:first}, Kontsevich-Zagier periods in Section~\ref{ssec:KZ}, cohomological periods in Section~\ref{ssec:coh},
and periods of motives in Section~\ref{ssec:mot_period}. 
In most cases, we will
define the subalgebra $\Per^\eff$ of effective periods and then
pass to $\Per:=\Per^\eff[\pi^{-1}]$. In particular, $\pi$ will be an element
of $\Per^\eff$.

Good to know:

\begin{metathm}All definitions describe the same set.
\end{metathm}

Throughout, we fix an embedding $\Qbar\to\C$. Let $\QR=\Qbar\cap \R$.
We denote by $\K$ a subfield of $\Qbar$.

\subsection{First definition}\label{ssec:first}
Let $X$ be a smooth algebraic variety over $\K$ of dimension $d$, $D\subset X$ a divisor with normal crossings (also defined over $\K$),
 $\omega\in\Omega^d_X(X)$ 
a global algebraic differential form, $\sigma=\sum a_i\gamma_i$ a
differentiable relative chain, i.e., $a_i\in\Q$, $\gamma_i:\Delta^d\to X^\an$
a $C^\infty$-map on the $d$-simplex such that $\partial \sigma$ is a chain
on $D^\an$. 
Then
\[ \int_\sigma\omega:=\sum a_i\int_{\Delta^d}\gamma_i^*\omega^\an\]
is a complex number.

\begin{defn}The set of these complex numbers is called the set of
effective periods $\Per^\eff$. 
\end{defn}
They appear as nc (for normal crossings) periods in \cite[Section~11.1]{period-buch}.
\begin{ex}The algebraic number $\alpha$ satisfying the non-trivial polynomial $P\in \K[t]$ is an effective period as
the value of the integral 
\[ \int_0^\alpha dt\]
with $X=\A^1$, $D=V(tP)$, $\omega=dt$, $\sigma=\gamma$ the standard
path from $0$ to $\alpha\in D(\C)$ in $\A^1(\C)=\C$.
\end{ex}

\begin{ex}[Cauchy's integral] The number $2\pi i$ is an effective period as the
value of the integral
\[ \int_{S^1}\frac{dz}{z}\]
with $X=\G_m$, $D=\emptyset$, $\omega= dz/z$, $\sigma=\gamma:[0,1]\to \C^*$ given by
$\gamma(t)=\exp(2\pi i t)$.
\end{ex}
\begin{lemma}The set $\Per^\eff$ is a $\Qbar$-algebra. It is independent
of the choice of $\K$.
\end{lemma}
\begin{proof}\cite[Proposition~11.1.7]{period-buch}.\end{proof}
In particular, $\pi \in \Per^\eff$.

\begin{defn}\label{defn:first}
\[ \Per:=\Per^\eff[\pi^{-1}].\]
\end{defn}

\begin{rem}
\begin{enumerate}
\item
In the above definition, it even suffices to restrict to affine $X$.
\item We may also relax the assumptions: $X$ need not be smooth, $D$ only
a subvariety, $\omega\in\Omega^n(X)$ not necessarily of top degree but only closed. 
\item We may also allow $\omega$ to have poles on $X$. We then need to impose the condition that the integral is absolutely convergent. This is the version given by Kontsevich and Zagier in \cite{kontsevich_zagier}.
\end{enumerate}
These implications are not obvious, but see \cite[Theorem~12.2.1]{period-buch}.
\end{rem}

\subsection{Semi-algebraic sets}\label{ssec:KZ}
The easiest example of a differential form is $dt_1\wedge\dots\wedge dt_n$
on $\A^n$ with coordinates $t_1,\dots,t_n$. Actually, this is the only one we need.

\begin{defn}\label{defn:KZ}An \emph{effective Kontsevich-Zagier period} is complex number whose real and imaginary part can be written as difference between the volumes of
$\Q$-semi-algebraic subsets of $\R^n$ (of finite volume).
\end{defn}
Note that a subset of $\R^n$ is $\Q$-semi-algebraic if and only if it is $\QR$-semi-algebraic.

\begin{prop}\label{prop:all_KZ} The set of effective Kontsevich-Zagier periods agrees with
$\Per^\eff$.
\end{prop}
\begin{proof}\cite[Prop.~12.1.6]{period-buch}, combined with \cite[Theorem~12.2.1]{period-buch}.
\end{proof}

\subsection{Cohomological periods}\label{ssec:coh}
The definition becomes more conceptual when thinking in terms of cohomology.
Chains define classes in (relative) singular homology and differential forms define classes in (relative) algebraic de Rham cohomology. The definition of algebraic de Rham cohomology is very well-known in the case of a single smooth algebraic variety. There are different methods of extending this to the singular and to the relative case, see \cite[Chapter~3]{period-buch}. Singular cohomology and algebraic de Rham cohomology become isomorphic after base change to the complex numbers. By adjunction this defines
the \emph{period pairing}.

Let $X/\K$ be an algebraic variety, $Y\subset X$ a subvariety. Let
\[ \int: H^i_\dR(X,Y)\times H_i^\sing(X,Y;\Q)\to \C\]
be the period pairing.

\begin{defn}\label{defn:coh_period}The numbers in the image of $\int$ (for varying $X$, $Y$, $i$)
are called \emph{effective cohomological periods}.
\end{defn}

\begin{prop}[{\cite[Theorem~11.4.2]{period-buch}}]\label{prop:all_coh}The set of effective cohomological periods agrees with
$\Per^\eff$.
\end{prop}

\subsection{Periods of motives}\label{ssec:mot_period}
The insistance that we need to invert $\pi$ in order to get a good object becomes
clear from the motivic point  of view.

Two kinds of motives are relevant for our discussions:
\begin{enumerate}
\item Voevodsky's triangulated categories of (effective) geometric motives
$\DMgm^\eff(\K,\Q)$ and $\DMgm(\K,\Q)$;
\item Nori's abelian category $\MMN(\K,\Q)$ of motives over $\K$ with respect to
singular cohomology and its subcategory of effective motives.
\end{enumerate}

\begin{thm}[Nori, Harrer {\cite[Theorem~7.4.17]{harrer}}] There is a natural contravariant trianguled tensor-functor
\[ \DMgm^\eff(\K,\Q)\to D^b(\MMN^\eff(\K,\Q))\]
compatible with the singular realisation. It extends to non-effective motives.
\end{thm}

All motives have singular and de Rham cohomology and hence periods. We formalise as in \cite[Section~11.2]{period-buch}.

\begin{defn}\label{defn:VV}
Let $\VV$ be the category of tuples $(V_\K,V_\Q,\phi)$ where
$V_\K$ is a  finite dimensional $\K$-vector space, $V_\Q$ a finite dimensional
$\Q$-vector space and $\phi:V_\K\tensor_\K\C\to V_\sing\tensor_\Q\C$ an isomorphism.
Morphisms are given by pairs of linear maps such that the obvious diagram commutes.

Given $V\in \VV$, we define the set of periods of $V$ as
\[ \Per(V)=\im (V_\K\times V_\Q^\vee\to \C),\hspace{2ex} (\omega,\sigma)\mapsto \sigma_\C(\phi(\omega))\]
and the space of periods $\Per\langle H\rangle$ as the additive group generated by it.
If $\Ch\subset \VV$ is a subcategory, we put
\[ \Per(\Ch)=\bigcup_{V\in\VV}\Per(V).\]
\end{defn}

\begin{rem}
\begin{enumerate}
\item If $\Ch$ is additive, then $\Per(\Ch)$ is a $\K$-vector space. 
In particular, $\Per\langle V\rangle$ agrees with $\Per(\langle V\rangle)$ where
$\langle V\rangle$ is the full abelian subcategory of $\VV$ closed under subquotients generated by $V$.

\item If $\Ch$ is an additive tensor category, then $\Per(\Ch)$ is a $\K$-algebra.
\end{enumerate}
\end{rem}

By the universal property of Nori motives, there is a faithful exact functor
\[ \MMN(\K)\to \VV.\]
This defines periods of Nori motives, and by composition with
$M\mapsto \bigoplus H^i_\Nori(M)$ also periods of geometric motives.

\begin{defn}\label{defn:mot_period}A complex number is called an \emph{(effective) motivic period}
if is the period of an (effective) motive.
\end{defn}
Every Nori motive is a subquotient of the image of a geometric motive, hence is
does not matter whether we think of geometric or Nori motives in the above definition.

\begin{prop}[{\cite[Proposition~11.5.3]{period-buch}}]\label{prop:all_mot}The set of effective motivic period agrees with $\Per^\eff$. The 
set of all motivic periods agrees with $\Per$. In particular, both are independent of the choice of $\K$.
\end{prop}

As a consequence, we see that both $\Per^\eff$ and $\Per$ are $\Qbar$-algebras.
Of course this is also easy to deduce directly from any of the other descriptions.

There is a big advantage of the motivic description: any type of structural result on motives has consequences for periods. One such are the bounds on the spaces of multiple zeta-values by Deligne-Goncharov. Another example are the transcendence results for $1$-motives as discussed below.

In particular, the motivic point of view allows us (or rather Nori) to give a structural reinterpretation of the period conjecture.

\subsection{Comparison}
We conclude by putting all information together:

\begin{thm}[{Proposition~\ref{prop:all_KZ}}, \ref{prop:all_coh}, \ref{prop:all_mot}]The following sets of complex numbers agree:
\begin{enumerate}
\item the set of periods in the sense of Definition~\ref{defn:first};
\item the set of Kontsevich-Zagier periods, see Definition~\ref{defn:KZ}; 
\item the set of cohomological periods, see Definition~\ref{defn:coh_period};
\item the set of all motivic periods, see Definition~\ref{defn:mot_period}.
\end{enumerate}
Moreover, they are independent of the field of definition
$\Q\subset\K\subset\Qbar$ and form a countable $\Qbar$-algebra.
\end{thm}

\section{Formal periods and the period conjecture}

In reverse order to the historic development, we start with Kontsevich's formulation and discuss the Grothendieck conjecture afterwards.

\subsection{Kontsevich's period conjecture}

We take the point of view of cohomological periods. At this point we restrict to $\K=\Qbar$ for simplicity. There are obvious  relations between periods:

\begin{itemize}
\item (bilinearity) Let $\omega_1,\omega_2\in H^i_\dR(X,Y)$, $\alpha_1,\alpha_2\in\Qbar$, $\sigma_1,\sigma_2\in H_i^\sing(X,Y;\Q)$, $a_1,a_2\in\Q$. Then
\[ \int_{\sigma_1+\sigma_2}(\alpha_1\omega_1+\alpha_2\omega_2)=\sum_{i,j=1}^2\alpha_j\int_{\sigma_i}\omega_j.\]
\item (functoriality) For $f:(X,Y)\to (X',Z)$, $\omega'\in H^i_\dR(X',Y')$,
$\sigma\in H_i^\sing(X,Y;\Q)$ we have
\[ \int_{f_*\sigma}\omega'=\int_\sigma f^*\omega'.\]
\item (boundaries) For $Z\subset Y\subset X$, $\partial:H^i_\dR(Y,Z)\to H^{i+1}(X,Y)$, $\delta:H_{i+1}^\sing(X,Y;\Q)\to H_i^\sing(Y,Z;\Q)$, $\omega\in H^i_\dR(Y,Z)$ and $\sigma\in H_{i+1}^\sing(X,Y;\Q)$ we have
\[ \int_{\delta\sigma}\omega=\int_\sigma\partial\omega.\]
\end{itemize}

\begin{defn}[Kontsevich]\label{defn:perf}The space of \emph{formal effective periods} $\Perf^\eff$ is defined
as the $\Q$-vector space generated by symbols 
\[ (X,Y,i,\omega,\sigma)\]
with $X$ an algebraic variety over $\Qbar$, $Y\subset X$ a closed subvariety, 
$i\in\Na_0$, $\omega\in H^i_\dR(X,Y)$, $\sigma\in H_i^\sing(X,Y;\Q)$ modulo
the space of relations generated by bilinearity, functoriality and boundary maps as above.

It is turned into a $\Q$-algebra with the product using the external product
induced by the K\"unneth decomposition
\[ H^*(X,Y)\tensor H^*(X',Y')\isom H^*(X\times X',Y\times X'\cup Y'\times X).\]

Let $2\tilde{\pi}i$ be the class of the symbol 
$(\Gm,\{1\},dz/z,S^1)$. The space of \emph{formal periods} $\Perf$ is
defined as the localisation
\[ \Perf=\Perf^\eff[\tilde{\pi}^{-1}].\]
\end{defn}

\begin{rem}
\begin{enumerate}
\item $\Perf^\eff$ and $\Perf$ are even $\Qbar$-algebras. We define the scalar multiplication by $\alpha(X,Y,i,\omega,\sigma)=(X,Y,i,\alpha\omega,\sigma)$ for
$\alpha\in\Qbar$. 
\item We do not know if it suffices to work with $X$ smooth, $D$ a divisor with normal crossings as Kontsevich does in \cite{kontsevich}. As explained in
 \cite[Remark~13.1.8]{period-buch} these symbols generate the algebra, but is is not clear if they also give all relations. Indeed, Kontsevich only imposes relations in an even more special case.
\item It is not at all obvious that the product is well-defined. This fact  follows from Nori's theory, see \cite[Lemma~14.1.3]{period-buch}.
\end{enumerate}
\end{rem}

There is a natural evaluation map
\[ \Perf^\eff\to\C,\quad (X,Y,i,\omega,\sigma)\mapsto \int_\sigma\omega.\]
By definition its image is the set of effective cohomological periods, hence equal to $\Per^\eff$.
\begin{conj}[Kontsevich]\label{conj:kont}
 the evaluation map
\[ \Perf\to\Per\]
is injective. In other words,
all $\Qbar$-linear relations between 
periods are induced by the above trivial relations.
\end{conj}

\subsection{Categorical and motivic description}\label{ssec:form_cat}
The above description is ``minimalistic'' in using a small set of generators and relations. The downside is that it is somewhat arbitrary. We now explain alternative descriptions. Let again $\K\subset\Qbar$ be a subfield. Recall that we have fixed an embedding $\Qbar\subset\C$.

\begin{defn}Let $\Ch$ be an additive category, $T:\Ch\to\VV$ an additive functor. We define the \emph{space of formal periods of $\Ch$} as the
$\Q$-vector space $\Perf(\Ch)$ generated by symbols
$(M,\sigma,\omega)$ with $M\in\Ch$, $\sigma\in T(M)_\Q^\vee$, $\omega\in T(M)_\K$ with relations
\begin{itemize}
\item bilinearity in $\sigma$ and $\omega$, 
\item functoriality in $M$.
\end{itemize}
\end{defn}

We are going to apply this mainly to subcategories of
$\MMN(\K,\Q)$ or $\DMgm(\K,\Q)$. 

\begin{thm}\label{thm:indep}
\[ \Perf(\DMgm(\K,\Q))=\Perf(\MMN(\K,\Q))=\Perf,\]
independent of $\K$.
\end{thm}
\begin{proof} This is only implicit in  \cite{period-buch}. It is a consequence of Galois theory of periods as discussed in the next section. We defer the proof. 
\end{proof}

\begin{rem}In particular, the validity of the period conjecture does not depend on $\K\subset \Qbar$. This is the reason that it was safe to
consider only $\K=\Qbar$ in Definition~\ref{defn:perf}.
\end{rem}

The evaluation map $\VV\to\C$ induces linear maps
\[ \Perf(\Ch)\to\C.\]
We denote the image by $\Per(\Ch)$.  

\begin{defn}We say that the period conjecture holds for $\Ch$ if
the map $\Perf(\Ch)\to\Per(\Ch)$ is injective.
\end{defn}
In the case $\Ch=\MMN(\K,\Q)$, we get back Kontsevich's conjecture \ref{conj:kont}.
We are going to study the relation between the period conjecture for $\Ch$ and subcategories of $\Ch$ below using a Galois theory for formal periods.

\section{Galois theory of (formal) periods}

	\subsection{Torsors and semi-torsors}
\begin{defn}\label{defn:semi}Let $k$ be a field.
\begin{enumerate}
\item 
 A \emph{torsor} in the category of affine $k$-schemes is an affine
scheme $X$ with an operation of an affine $k$-group $G$
\[ G\times X\to X\]
such that there is an affine $S$ over $k$ and a point $x\in X(S)$ such that the 
morphism 
\[ G_S\to X_S,\quad g\mapsto gx\]
induced by the operation
is an isomorphism.
\item An \emph{semi-torsor} over $k$ is a $k$-vector space $V$ together with a coalgebra $A$ and a right co-operation
\[ V\to V\tensor A\]
such there is a a commutative $k$-algebra $K$ and a $k$-linear map $V\to K$ such that the map \[ V_K\to A_K\]
induced by the comodule structure
is an isomorphism.
\end{enumerate}
\end{defn}

\begin{rem}\label{rem:stor}
\begin{enumerate}
\item 
The isomorphism makes $X$ a torsor in the fpqc-to\-po\-logy because $k\to K$ is faithfully flat. See also \cite[Section~1.7]{period-buch} for an extended discussion of different notions of torsors.
\item \label{it:stor_tor}
The second notion is newly introduced here. A torsor gives rise to a semi-torsor by taking global sections. A semi-torsor is a torsor, if, in addition,
$A$ is a Hopf algebra, $V$ a $k$-algebra and the structure map is a ring homomorphism. Hence a semi-torsor only has half of the information of a torsor. Note that an operation $G\times X\to X$ that gives rise to a semi-torsor, is already a torsor because an algebra homomorphism which is a vector space isomorphism is also an algebra isomorphism.
\item If $B$ is a finite dimensional $k$-algebra, $N$ a left-$B$-module that is fpqc-locally free of rank $1$ (i.e., here is  a commutative $k$-algebra $K$ such that $K\tensor_kN$ is free of rank $1$), then $N^\vee:=\Hom_k(N,k)$ is a semi-torsor over the coalgebra $B^\vee$.
\item Once there is one $x$ satisfying the torsor condition, all elements of $X(S)$ for all $S$ will satisfy the condition. This is not the case for semi-torsors.
\end{enumerate}
\end{rem}

\subsection{Galois theory revisited}\label{ssec:galois}
As promised in the introduction, we want to formulate Galois theory in torsor language.

Let $\Gamma=\Gal(\Qbar/\Q)$. It is profinite, hence proalgebraic where we view abstract finite groups as algebraic groups over $\Q$. Indeed, we identify a 
finite group $G$ with the algebraic group given by $\Spec( \Q[G]^\vee)$. 
Let $X_0=\Spec(\Qbar)$. We view it as a pro-algebraic variety over $\Q$.
There is a natural operation
\[ \Gamma\times X_0\to X_0.\]
We describe it on finite level. Let $K/\Q$ be Galois. Then
\[ \Gal(K/\Q)\times \Spec (K)\to \Spec( K)\]
is given by the ring homomorphism
\[ \Q[\Gal(K/\Q)]^\vee\tensor K\leftarrow K\]
adjoint to
\[ K\leftarrow \Q[\Gal(K/\Q)]\tensor K\]
induced by the defining operation of the abstract group $G(K/\Q)$ on $K$.

\begin{thm}[Galois theory]
The space $X_0$ is a torsor under $\Gamma$. More precisely,
the choice $\id\in X_0(\Qbar)$ induces an isomorphism
\[ \Gamma_\Qbar\to (X_0)_\Qbar.\]
\end{thm}
\begin{proof}
We check the claim on the isomorphism on finite level.
Let $K/\Q$ be a finite Galois extension. The projection $X_0\to \Spec(K)$
amounts to fixing an embedding $K\subset\Qbar$.

We want to study the base change of
\[\Gal(K/\Q)\times \Spec(K)\to\Spec(K)\]
to $\Qbar$. On rings, this is
\begin{equation}\label{eq1} \Qbar[\Gal(K/\Q)]^\vee\tensor_\Qbar (K\tensor\Qbar)\leftarrow K\tensor\Qbar.\end{equation}
On the other hand, consider the map 
\[ K\tensor \Qbar\to \mathrm{Map}(\Hom(K,\Qbar),\Qbar)\isom \Qbar[\Hom(K,\Qbar)]^\vee,\quad \lambda\tensor a\mapsto (\sigma\mapsto a\sigma(\lambda)).\]
It is Galois equivariant. As $K/\Q$ is separable, it is an isomorphism.
Hence our map (\ref{eq1}) can be written as
\begin{equation}\label{eq2}
\Qbar[\Gal(K/\Q)]^\vee\tensor \Qbar[\Hom(K,\Qbar)]^\vee\leftarrow \Qbar[\Hom(K,\Qbar)]^\vee.\end{equation}
It is obtained by duality from the right operation
\begin{equation} \Gal(K/\Q)\times\Hom(K/\Qbar)\to\Hom(K/\Qbar).\end{equation}
By Galois theory, this operation is simply transitive. By evaluating
on the fixed inclusion $K\subset \Qbar$, we get a bijection
\[  \Gal(K/\Q)\to\Hom(K,\Qbar)\]
translating to the isomorphism
\[ \Qbar[\Gal(K/\Q]^\vee\leftarrow K\tensor \Qbar\]
that we needed to show.
\end{proof}

\subsection{Nori's semi-torsor}
We return to the situation of Section~\ref{ssec:form_cat}. We want to explain how formal periods are a semi-torsor or even torsor.

Let $\K\subset \C$,
$\Ch$ an additive category,
\[ T:\Ch\to \VV,\quad M\mapsto (T(M)_\K, T(M)_\Q,\phi_M)\]
an additive functor. From now on we write
\[ T_\dR(M):= T(M)_\K, \hspace{2ex}T_\sing(M):=T(M)_\Q.\]
Both $T_\dR$ and $T_\sing$ are additive functors to $\K\Vect$ and $\Q\Vect$, respectively. In the Tannakian spirit we call
them \emph{fibre functors}.

We first give an alternative description of the formal period space. For a fixed $M$, the $\Q$-vector space generated by the symbols $(M,\sigma,\omega)$ with the relation of bilinearity is nothing but
\begin{multline*}
 T_\dR(M)\tensor_\Q T_\sing(M)^\vee\isom T_\dR(M)\tensor_\K T_\sing^\vee(M)_\K\\
\isom \Hom_\K(T_\dR(M),T_\sing(M)_\K)^\vee.
\end{multline*}
Hence we have an alternative description of the space of formal periods
\[ \Perf(\Ch)=\left(\bigoplus_{M\in\Ch} \Hom_\Qbar(T_\dR(M),T_\sing(M))^\vee\right)/\text{functoriality}.\]
Note that the $\Q$-sub vector space of relations induced by functoriality is 
even a $\K$-vector space. This makes clear that $\Perf(\Ch)$ is a $\K$-vector space. 

In the next step, we introduce the coalgebra (or even affine $\K$-group) under which the
formal period space is a semi-torsor (or even torsor). Actually, its definition is parallel to the definition of the formal period space, but with one fibre functor instead of two.

\begin{defn}For $M\in\Ch$ put
\[ \Ah(M)=\End_\Q(T_\sing(M))^\vee\]
and
\[ \Ah(\Ch)=\left(\bigoplus_{M\in\Ch}\Ah(M)\right)/\text{functoriality}.\]
\end{defn}
Note that $\Ah(M)$ and $\Ah(\Ch)$ are $\Q$-coalgebras. 

\begin{rem}This object is at the very heart of Nori's work. It is an
 explicit description of the coalgebra that Nori  attaches to the additive functor $T_\sing:\Ch\to  \Q\Vect$. 
\end{rem}

We extend scalars to
$\K$:
\[ \Ah(\Ch)_\K=\left(\bigoplus_{M\in\Ch}\End_\K(T_\sing(M)_\K)\right)^\vee/\text{functoriality}.\]

The left-module structure
\begin{multline*}
 \End_\K(T_\sing(M)_\K,T_\sing(M)_\K)\times \Hom_\K(T_\dR(M),T_\sing(M)_\K)\\
\longrightarrow \Hom_\K(T_\dR(M),T_\sing(M)_\K)
\end{multline*}
induces a right-comodule structure 
\[ \Hom_\K(T_\dR(M),T_\sing(M)_\K)^\vee\to \Ah(M)_\K\tensor \Hom_\K(T_\dR(M),T_\sing(M)_\K)^\vee \]
and hence
\[ \Perf(\Ch)\to\Ah(\Ch)_\K\tensor_\K\Perf(\Ch)\isom \Ah(\Ch)\tensor_\Q\Perf(\Ch).\]
Recall that $T$ takes values in $\VV$, hence we have a distinguished 
isomorphism of functors $\phi_T:T_\dR(M)_\C\to T_\sing(M)_\C$. It is an element
of $\Hom_\C(T_\dR(M)_\C,T_\sing(M)_\C)$ or equivalently a linear map
\[ \Hom_\K(T_\dR(M),T_\sing(M)_\K)^\vee \to\C.\] Hence, this is precisely the data that we need for a semi-torsor. Indeed, because $\phi_T$ is an isomorphism,
 it induces
an isomorphism
\[ \Perf(\Ch)_\C\to \Ah(\Ch)_\C.\]
We have shown:
\begin{prop}The space $\Perf(\Ch)$ is a $\K$-semi-torsor over $\Ah(\Ch)_\K$ in the sense of Definition~\ref{defn:semi}.
\end{prop}

\begin{defn}Let $\CDTarg{\Ch}{T_\sing}$ be the category of $\Ah(\Ch)$-comodules finite dimensional over $\Q$.
\end{defn}
This is nothing but Nori's diagram category, see \cite[Section~7.1.2]{period-buch}.

\begin{thm}[Nori, see {\cite[Theorem~7.1.13]{period-buch}}] There is a natural factorisation
\[ T_\sing:\Ch\to\CDTarg{\Ch}{T_\sing}\xrightarrow{\tilde{T}_\sing}\Q\Vect\]
with $\CDTarg{\Ch}{T_\sing}$ abelian and $\Q$-linear, the functor $\tilde{T}$ faithful and exact, and the category is universal with this property.
\end{thm}

\begin{rem}\label{rem:univ}In particular, $\Ch$ is equivalent to $\CDTarg{\Ch}{T_\sing}$ if
$\Ch$ is itself abelian and $T_\sing$ faithful and exact.
\end{rem}
\begin{prop}\label{prop:make_ab}We have
\[ \Perf(\Ch)\isom\Perf(\CDTarg{\Ch}{T_\sing}).\]
\end{prop}
\begin{proof}We use the semi-torsor structure to reduce the claim to
the comparison of the diagram algebra. It remains to check that the natural map
\[ \Ah(\Ch)\to\Ah(\CDTarg{\Ch}{T_\sing})\]
is an isomorphism. This is only implicit in \cite[Section~7.3]{period-buch}.
Here is the argument: By construction, $\Ah(\Ch)=\lim_i A_i$ for
coalgebras finite dimensional over $\Q$ of the form $A_i=E_i^*$ for a
$\Q$-algebra $E_i$. We also have $\CDTarg{\Ch}{T_\sing}=\bigcup_i E_i\Mod$, the category of finitely generated $E_i$-modules.
Hence it suffices to show that
\[ E_i\isom \Ah(E_i\Mod)^\vee.\]
The category $E_i\Mod$ has the projective generator $E_i$. By \cite[Lemma~7.3.14]{period-buch}, we have $\Ah(E_i\Mod)^\vee\isom\End(T_\sing|_{E_i})$. The latter means all $\Q$-linear maps $E_i\to E_i$ commuting with all $E_i$-morphisms
$E_i\to E_i$. We have $\End_{E_i\Mod}(E_i)=E_i^\circ$ (the opposite algebra of
$E_i$) and hence $\End(T_\sing|_{E_i})=E_i$ as claimed.
\end{proof}

\subsection{The Tannakian case}
Assume now that, in addition, $\Ch$ is a rigid tensor category and
$T$ a faithfully exact tensor functor. Using $T_\sing$ as fibre functor, this makes $\Ch$ a \emph{Tannakian} category. By Tannaka duality, it is
equivalent to the category of finite dimensional representations of a 
pro-algebraic group $G(\Ch)$. There is a second fibre functor
given by $T_\dR$. Again by Tannaka theory, the comparison of the two
fibre functors defines a pro-algebriac affine scheme $X(\Ch)$ over $\K$ such that
\[ X(S)=\left\{\Phi:T_\sing(\cdot)_S\to T_\dR(\cdot)_S| \text{isom. of tensor functors}\right\}.\]
It is a torsor under $G(\Ch)$
\[ G(\Ch)\times X(\Ch)\to X(\Ch).\]
When passing to the underlying semi-torsor, these objects are identical to the ones considered before.

\begin{prop}[{\cite[Section~7.1.4]{period-buch}}]
There are natural isomorphisms
$\Oh(X)\isom \Perf(\Ch)$ and $\Oh(G(\Ch))=\Ah(\Ch)$.
\end{prop}
\begin{proof}The case of the group is \cite[Theorem~7.1.21]{period-buch}. 
The same arguments also applies  to $X(\Ch)$, see also \cite[Remark~8.2.11]{period-buch}.
\end{proof} 

\subsection{The case of motives}\label{ssec:motive_torsor}
A particularly interesting case ist $\Ch=\MMN(\K)$. Put
\begin{align*}
 G_\mot(\K)&:=\Spec\left(\Ah(\MMN(\K))\right),\\
 X(\K)&:=\Spec\left(\Perf(\MMN(\K))\right).
\end{align*}
Then
\[ G_\mot(\K)_\K\times_\K X(\K)\to X(\K)\]
is a torsor. We use this to give the proof of Theorem~\ref{thm:indep} that was left open.

\begin{proof}[Proof of Theorem~\ref{thm:indep}]
Fix $\K\subset\Qbar$. Let $\Ah_\Nori(\K)$ be the coalgebra defined
using the same data as in Definition~\ref{defn:perf}, but with $H^i_\dR$ replaced by $H^i_\sing$ and with base field $\K$ instead of $\Qbar$.

The considerations of the present section also apply to this case and show
that $\Ah_\Nori(\K)$ is nothing but Nori's diagram coalgebra for the diagram
of pairs, see \cite[Definition~9.1.1]{period-buch}. By definition, see
\cite[Definition~9.1.3]{period-buch}, the category $\MMN(\K)$ is the category
of representations of $G_\mot(\K)=\Spec\left(\Ah_\Nori(\K)\right)$. 
Hence $\Ah_\Nori(\K)\isom\Ah(\MMN(\K))$. 

By \cite[Corollary~10.1.7]{period-buch}
we also have $\Ah_\Nori(\K)\isom\Ah(\DMgm(\K))$. By Galois theory of periods, i.e., the torsor property, this also implies
\[ \Perf(\K)\isom\Perf(\MMN(\K))\isom\Perf(\DMgm(\K)).\]

We now have to compare different $\K$'s. Let $K/k$ be an algebraic extension of
subfields of $\Qbar$. The comparison
\[ \Perf(K)\isom\Perf(k)\]
is claimed in \cite[Proposition~3.1.11]{period-buch}. Unfortunately, the argument is not complete. It only shows that the map induced by base change
\[ \Per(k)\to\Perf(K)\]
is surjective. We are going to complete the argument now. It suffices to consider the case $K/k$ finite and Galois. In this case, there is a natural short exact sequence
\[ 0\to G_\mot(K)\to G_\mot(k)\to\Gal(K/k)\to 0\]
by \cite[Theorem 9.1.6]{period-buch}. Moreover, we have torsor structures
\begin{gather*}
G_\mot(k)_K\times_K X(k)_K\to X(k)_K,\\
G_\mot(K)_K\times_K X(K)\to X(K).
\end{gather*}
We have $\Qbar\subset\Perf(K)$ as formal periods of zero-dimensional varieties.
Hence the structure morphism $X(k)\to \Spec(k)$ factors via
$\Spec(K)$. We write $X'(k)$ for $X(k)$ viewed as $K$-scheme. Then
\[ X(k)_K=X(k)\times_kK\isom X'(k)\times_K(\Spec(K)\times_k\Spec(K))\isom\coprod_{\sigma\in\Gal(K/k)}X'(k).\]
Comparing with the structure of $G_\mot(k)$, we deduce the torsor
\[ G_\mot(K)_K\times_KX'(k)\to X'(k).\]
Hence the natural map $X'(k)\to X(K)$ is an isomorphism.
\end{proof}

\section{Consequences for the period conjecture}
\subsection{Abstract considerations}
As before,
let $\Ch$ be an additive category and $T:\Ch\to \VV$ an additive functor
and $\CDTarg{\Ch}{T_\sing}$ the diagram category.
\begin{prop}
The period conjecture for $\Ch$ is equivalent to the period conjecture for
$\CDTarg{\Ch}{T_\sing}$.
\end{prop}
\begin{proof}The formal period algebras agree by
Proposition~\ref{prop:make_ab}.
\end{proof}
Hence we only need to consider abelian categories and faithful exact functors
$T$ when analysing the period conjecture.

\begin{prop}\label{prop:inj}Let $\Ch$ be an abelian, $\Q$-linear category. Let $\Ch'\subset \Ch$ be an abelian subcategory such that the inclusion is exact. The following are equivalent:
\begin{enumerate}
\item \label{it:1}the category $\Ch'$ is closed under subquotients in $\Ch$;
\item \label{it:2}the map $\Perf(\Ch')\to\Perf(\Ch)$ is injective.
\end{enumerate}
\end{prop}
\begin{proof}Injectivity can be tested after base change to $\C$, hence the second assertion of equivalent to the injectivity of $\Ah(\Ch')\to\Ah(\Ch)$.
The implication from (\ref{it:1}) to (\ref{it:2}) is \cite[Proposition~7.5.9]{period-buch}. 

For the converse, consider $K=\ker(\Ah(\Ch')\to\Ah(\Ch)$. It is an $\Ah(\Ch)$-comodule, but not a $\Ah(\Ch')$-comodule. 
By construction
$\Ah(\Ch')$ is a direct limit of coalgebras finite over $\Q$. Hence
$K$ is a direct limit of $\Ah(\Ch)$-comodules of finite dimension over $\Q$.
The cannot all be $\Ah(\Ch')$-comodules, hence we have found objects of
$\Ch$ that are not in $\Ch'$.
\end{proof}
\begin{cor}\label{cor:faithful}Suppose $\Ch$ is abelian and $T$ faithful and exact. If the period conjecture holds
for $\Ch$, then $T:\Ch\to \VV$ is fully faithful with image closed
under subquotients.
\end{cor}
\begin{proof}We view $\Ch$ as a subcategory of $\VV$ via $T$. 
Let $\tilde{\Ch}$ be its closure under
subquotients. Note that the periods of $\Ch$ agree with the periods
of $\tilde{\Ch}$. By assumption, the composition
\[ \Perf(\Ch)\to \Perf(\tilde{\Ch})\to\Per(\Ch)\]
is injective. Hence the first map injective. By Proposition~\ref{prop:inj} this implies
that $\Ch$ is closed under subquotients in $\VValg$. Let $X,Y\in\Ch$ and
$f:X\to Y$ a morphism in $\VValg$. Then its graph $\Gamma\subset X\times Y$
is a subobject in $\VValg$, hence in $\Ch$. This implies that $f$ is in
$\Ch$. Hence $\Ch$ is a full subcategory.
\end{proof}

\begin{cor}\label{cor:full}Let $\Ch'\subset\Ch$ be a full abelian subcategory closed under subquotients. If the period conjecture holds for $\Ch$, then it holds for $\Ch'$.
\end{cor}
\begin{proof}We have $\Perf(\Ch')\subset\Perf(\Ch)\to\C$. If the second
map is injective, then so is the composition.
\end{proof}

\subsection{The case of motives}
We now specialise to $\Ch=\MMN(\K,\Q)$.

\begin{defn}Let $M\in \MMN(\K,\Q)$. We denote by $\langle M\rangle$ the full
abelian category of $\MMN(\K,\Q)$ closed under subquotients generated by $M$.
\end{defn}

\begin{prop}\label{prop:equiv1}The following are equivalent:
\begin{enumerate}
\item 
The period conjecture holds for $\MMN(\K,\Q)$.
\item  The period conjecture holds for $\langle M\rangle $ for all $M\in\MMN(\K,\Q)$.
\end{enumerate}
\end{prop}
Note that $\Per(\langle M\rangle)=\Per\langle M\rangle$ in the notation of
 Definition~\ref{defn:VV}.
\begin{proof}We have $\Perf(\langle M\rangle)\subset\Perf(\MMN)$ because the subcategory is full and closed under subquotients. Hence
\[ \Perf=\Perf(\MMN)=\bigcup_M\Perf(\langle M\rangle).\]
The map $\Perf\to\C$ is injective if and only if this is true for all 
$\Perf(\langle M\rangle)$.
\end{proof}

\subsection{Grothendieck's period conjecture}
In the previous section, we did not make use of the tensor product on motives.
Indeed, by Section~\ref{ssec:motive_torsor} we have the torsor
\[ G_\mot(\Qbar)_\Qbar\times_\Qbar X(\Qbar)\to X(\Qbar).\]

\begin{cor}[Nori]\label{galois}If the period conjecture holds, then $\Per$ is a $G_\mot(\K)_\K$-torsor.
\end{cor}

\begin{defn}Let $M\in\MMN(\Qbar)$. We put $\langle M\rangle^{\tensor, \vee}\subset \MMN(\Qbar)$ the full abelian rigid tensor subcategory closed under subquotients generated by $M$. We denote $G(M)$ its Tannaka dual.
\end{defn}

\begin{prop}\label{prop:equiv2}The following are equivalent:
\begin{enumerate}
\item the period conjecture holds for $\MMN(\Qbar)$;
\item the period conjecture holds for $\langle M\rangle^{\tensor,\vee}$ for
all $M\in\MMN(\Qbar)$.
\end{enumerate}
\end{prop}
\begin{proof}Same proof as for Proposition~\ref{prop:equiv1}\end{proof}

The scheme $X(M):=\Spec(\Perf\langle M\rangle^{\tensor,\vee})$
is a $G(M)$-torsor. Hence it inherits all properties of the algebraic group $G(M)$ that can be tested after a faithfully flat base change. In particular it is smooth.

\begin{conj}[Grothendieck's period conjecture]
Let $M\in\MMN$. Then $X(M)$ is
connected and
\[ \dim G(M)=\trdeg \Q( \Per(M)).\]
\end{conj}
In the case of pure motives, this is the formulation of Andr\'e in \cite[Chapitre~23]{andre2}.

\begin{prop}[{\cite[Conjecture~13.2.6]{period-buch}}]The following are equivalent:
\begin{enumerate}
\item the period conjecture holds for $\langle M\rangle^{\tensor,\vee}$;
\item the point $\mathrm{ev}\in X(M)(\C)$ is generic and $X(M)$ is connected;
\item Grothendieck's period conjecture holds for $M$.
\end{enumerate}
\end{prop}
\begin{proof}
See the proof of equivalence in \cite{period-buch}. The crucial input is
that $X(M)$ is smooth. Hence being connected makes it integral, so that there is only one generic point. Moreover, $X(M)$ and $G(M)$ have the same dimension.
\end{proof}

\section{Cycles and the period conjecture}
We now discuss the relation between conjectures on fullness of cycle class maps
and the period conjecture. 
\begin{prop}Let $\Ch\subset \MMN(\K)$ be a full abelian category closed under subquotients. If the period conjecture holds
for $\Ch$, then $H:\Ch\to \VV$ is fully faithful with image closed
under subquotients.
\end{prop}
\begin{proof}This is a special case of Corollary \ref{cor:full}.
\end{proof}
In the Tannakian case this is already formulated in \cite[Proposition~13.2.8]{period-buch}. Indeed, it seems well-known. 

By definition $H$ is faithful, hence the interesting part about this is
fullness. This point is taken up in a version of the period conjecture
formulated by Bost and Charles in \cite{bost-charles} and generalised
by Andreatta, Barbieri-Viale and Bertapelle in \cite{ABB}. We  formulate the latter version.
Recall from \cite{harrer} the realisation functor
\[ \DMgm(\K)\to D^b(\MMN(\K)).\]
By composition with $\MMN(\K)\to \VV$ this gives us a functor
\[ H^0:\DMgm(\K)\to \VV.\]

\begin{defn}We say that the fullness conjecture holds
for $\Ch\subset \DMgm(\K)$ if
\[ H^0:\Ch\to\VV\]
is full.
\end{defn}

\begin{conj}[ {\cite[Section~1.3]{ABB}}]\label{conj:ABB}
Let $\Ch\subset \DMgm(\K)$ be a full additive subcategory. Then the fullness conjecture holds for $\Ch$.
\end{conj}
\begin{rem}
\begin{enumerate}
\item
Andreatta, Barbieri-Viale and Bertapelle also consider refined integral information. We do not go into this here. 
\item This terminology deviates from \cite{ABB} where the term period conjecture is used. 
\end{enumerate}
\end{rem}

Arapura established in \cite{arapura} (see also  \cite[Theorem~10.2.7]{period-buch}) that the category of pure Nori motives is equivalent to Andr\'es's abelian category $\AM$ of motives via motivated cycles.

\begin{prop}If the period conjecture holds for pure Nori motives and
the K\"unneth and Lefschetz standard conjectures are true, then the fullness conjecture
holds for Chow motives.
\end{prop}
\begin{proof}
The functor from Chow motives to Grothendieck motives is full by definition.
The K\"unneth and Lefschetz standard conjectures together  imply that the category of
Grothendieck motives is abelian and agrees with $\AM$, hence with the
category of pure Nori motives. By the period conjecture, the functor
to $\VValg$ is full. Together this proves fullness of the composition.
\end{proof}

A similar result actually holds for the full category of geometric motives. Its assumptions are very strong, but are indeed expected to be true.

\begin{prop}\label{prop:full_gm}We assume
\begin{enumerate}
\item the period conjecture holds for all Nori motives;
\item the category $\DMgm(\K)$ has a $t$-structure compatible with the singular realisation;
\item  the singular realisation is conservative;
\item the category $\DMgm(\K)$ is of cohomological dimension $1$.
\end{enumerate}
Then the fullness conjecture holds
for $\DMgm(\K)$.
\end{prop}
\begin{proof}The period conjecture gives again fullness of $\MMN(\K)\to\VV$.
It remains to check surjectivity of
\begin{equation}\label{eq:surj}\tag{*} \Hom_{\DMgm(\K)}(M_1,M_2)\to \Hom_{\MMN(\K)}(H^0_\Nori(M_1),H^0_\Nori(M_2))\end{equation}
for all $M_1,M_2\in\DMgm(\K)$.

Let $\MM$ be the heart of the $t$-structure. By conservativity, $H^0$ is faithful. By the universal property of Nori motives, this implies $\MM\isom \MMN(\K)$. As the cohomological dimension is $1$, this implies
$\DMgm(\K)\isom D^b(\MMN(\K))$. 
It also implies that every object of $\DMgm(\K)$ is (uncanonically) isomorphic to the direct sum of its cohomology objects. Hence it suffices to consider
$M_1=N_1[i_1]$, $M_2=N_2[i_2]$ for $N_1,N_2\in\MM$, $i_1,i_2\in\Z$. We have $H^0(M_k)=0$ unless
$0=i_1,i_2$. Hence surjectivity in (\ref{eq:surj}) is trivial for $i_1\neq 0$ or $i_2\neq 0$.  
In the remaining case $i_1=i_2=0$, we have equality in equation (\ref{eq:surj}), so again, the map is surjective.
\end{proof}

\begin{rem}The converse is not true, the fullness conjecture does \emph{not} imply
the period conjecture. An explicit example is given in \cite[Remark~5.6]{huber-wuestholz}. See also the case of mixed Tate motives, Proposition~\ref{prop:mt}.

On the other hand, the condition on the cohomological dimension is necessary if we want to conclude from fully faithfulness on the level of abelian categories to the triangulated category. Note that $\VV$ has cohomological dimension $1$. 
\end{rem}

\section{Examples}
\subsection{Artin motives}
The category of Artin motives over $\K$ can be described as the full subcategory of
$\MMN(\K)$ generated by motives of $0$-dimensional varieties. The Tannakian
dual of this category is nothing but $\Gal(\bar{\K}/\K)$. The space of
periods is $\Qbar$. The torsor structure is the one described in Section~\ref{ssec:galois}. For more details see also \cite[Section~13.3]{period-buch}.

\begin{rem}The period conjecture and the fullness conjecture hold true for Artin motives.
\end{rem}

\subsection{Mixed Tate motives}
Inside the triangulated category $\DMgm(\Spec(\Q),\Q)$ let $\DMT$ the full thick rigid tensor triangulated subcategory generated by the Tate motive $\Q(1)$. Equivalently, it is the full thick triangulated subcategory generated by the set of $\Q(i)$ for $i\in\Z$.
By deep results of Borel it carries a $t$-structure, so that we get a well-defined \emph{abelian} category $\MT$ of mixed Tate motives over $\Q$, see \cite{LeMTM}. Deligne and Gochanarov find a smaller abelian tensor subcategory $\MTZ$ of 
\emph{unramified} mixed Tate motives, see \cite{deligne-goncharov}. Its distinguishing feature is that
\[ \Ext^1_{\MTZ}(\Q(0),\Q(1))=\Z^*\tensor_\Z\Q=0\subsetneq \Ext^1_{\MT}(\Q(0),\Q(1))=\Q^*\tensor_\Z\Q.\]
By \cite{deligne-goncharov}, the periods of $\MTZ$ are precisely the famous multiple zeta values. The Tannaka dual of $\MTZ$ is very well-understood.
Using Galois theory of periods this information translates into information
on $\Perf(\MTZ)$. The surjectivity of $\Perf(\MTZ)\to \Per(\MTZ)$ then yields
upper bounds on the dimensions of spaces of multiple zeta values. For more
details see \cite{deligne-goncharov} or \cite[Chapter~15]{period-buch}.

\begin{rem}
The period conjecture and the 
fullness conjecture are open for $\MTZ$.
\end{rem}
Indeed, the period conjecture for $\MTZ$ implies statements like the algebraic independence of the values of Riemann $\zeta$-function $\zeta(n)$ for $n\geq 3$ odd that are wide open. The fullness conjecture is a lot  weaker.

\begin{prop}\label{prop:mt}The fullness conjecture holds for $\DMT$ if and only if $\zeta(n)$ is irrational for $n\geq 3$ odd.
\end{prop}
\begin{proof}Assume the irrationality condition. By \cite{deligne-goncharov}, the category $\MT$ has cohomological dimension $1$ and $D^b(\MT)$ is equivalent to $\DMT$. Hence we can use the same reasoning as in the proof of Proposition~\ref{prop:full_gm}. It suffices to show that $\MT\to \VVQ$ is fully faithful. We address this issue with the same argument as in the proof of \cite[Proposition~2.14]{deligne-goncharov} treating the functor from $\MTZ$ to mixed realisations. It remains to check injectivity of
\[ \Ext^1_{\MT}(\Q,\Q(n))\to \Ext^1_{\VVQ}(\Q,\Q(n))\]
for all $n\in\Z$. Here $\Q(n)$ denotes the object of rank $1$ given by
$\Q(n)_\dR=\Q, \Q(n)_\sing=(2\pi i)^n\Q$ and the natural comparison isomorphism.

By Borel's work, the left hand side is known to vanish for $n\leq 0$ and $n\geq 2$ even, hence it suffices to consider $n=2m-1$ for $m\geq 1$. By \cite[Proposition~4.3.2 and Proposition~4.2.3]{huber_1604}
\[ \Ext^1_{\VVQ}=\coker\left( \Q(n)_\dR\oplus \Q(n)_\sing\to \C\right)=\C/\langle 1,(2\pi i)^n\rangle_\Q .\]
This is a quotient of
\begin{equation}\label{eq:ext} \Ext^1_{\MHS_\Q}(\Q,\Q(n))\isom \C/(2\pi i)^n\Q.\end{equation}
We first consider $n=m=1$, the Kummer case. The left hand side is
$\Q^*\tensor\Q$ and the map is the Dirichlet regulator $u\mapsto \log|u|$.
The prime numbers give a basis of $\Q^*\tensor\Q$. By Baker's theorem, their
logarithms are $\Qbar$-linearly independent. They are also real whereas $2\pi i$ is imaginary. Hence the regulator map (\ref{eq:ext})
is injective.

Now let $m\geq 2$.
From the well-known explicit computation of the Deligne regulator, we know that the image of a generator of
$\Ext^1_{\MT}(\Q,\Q(2m-1))$ is given by $\zeta(2m-1)\in \C/\langle 1,(2\pi i)^{2m-1}\rangle$. It is real hence $\Q$-linearly independent of $(2\pi i)^{2m-1}\in i\R$. If it is irrational, then it is $\Q$-linearly independent of $1$ as well.

Now assume conversely that the fullness conjecture holds for $\DMT$. This implies that $\MT\to \VVQ$ is fully faithful. We claim that (\ref{eq:ext}) is injective. Take an element on the right, i.e., a short exact sequence
\[ 0\to \Q(n)\to E\to \Q(0)\to 0\]
in $\MT$ whose image in $\VVQ$ is split. This means that there is a splitting morphism $\Q(0)\to E$ in $\VVQ$. By fullness it comes from a morphism in $\MT$. This shows injectivity of (\ref{eq:ext}). By the explicit computation this implies
that $\zeta(2m-1)$ is $\Q$-linearly independent of $1$, i.e., irrational.
\end{proof}

We see that the fullness statement is much weaker that than the period conjecture. There are even partial results due to Ap\'ery and Zudilin.
\subsection{The case of $1$-motives}
The situation for $1$-motives is quite similar to the case of $0$-motives: 
there is an abelian category $\onemot_\K$ of iso-$1$-motives over $\K$. It was originally defined by Deligne, see \cite{hodge3}. Its derived category
is equivalent the subcategory of $\DMgm(\K)$ generated by motives
of varieties of dimension at most $1$. This result is due to Orgogozo, see \cite{orgogozo}, see also Barbieri-Viale and Kahn \cite{BVK}. There is a significant difference to the $0$-dimensional case, though: $\onemot_\K$ is not closed under tensor products, so it does not make sense to speak about torsors. However, our notion of a semi-torsor is built to cover this case: $\Perf(\onemot_\K)$ is semi-torsor
under the coalgebra $\Ah(\onemot_\K)_\K$. 

The main result of \cite{huber-wuestholz} settles the period conjecture:

\begin{thm}[{\cite{huber-wuestholz}}]\label{thm:1-mot}
The evalution map $\Perf(\onemot_\Qbar)\to\Perf(\onemot_\Qbar)$ is injective, i.e., the period conjecture holds for $\onemot_\Qbar$.
\end{thm}

This is first and foremost a result in transcendence theory. With a few extra arguments, it implies:

\begin{cor}[{\cite[Theorem~5.10]{huber-wuestholz}}]Let $X$ be  an algebraic curve over $\Qbar$, $\omega\in\Omega^1_{\Qbar(X)}$ an algebraic differential form, $\sigma=\sum_{i=1}^na_i\gamma_i$ a singular chain avoiding the poles of $\omega$ and with $\partial\sigma$ a divisor
on $X(\Qbar)$. Then the following are equivalent:
\begin{enumerate}
\item the number $\int_\sigma\omega$ is algebraic;
\item we have $\omega= df +\phi$ with $\int_\phi\phi=0$.
\end{enumerate}
\end{cor}

This contains the previously known results on transcendence of $\pi$ or
values of $\log$ in algebraic numbers as well as periods of differential forms over closed paths. The really new case concerns differential forms of the third kind (so with non-vanishing residues) and non-closed paths. This settles
questions of Schneider that had been open since the 1950s.
Here is a very explicit example:

\begin{ex}[{\cite[Theorem~7.6]{huber-wuestholz}}]Let $E/\Qbar$ be an elliptic curve. Recall the Weierstrass $\wp$-, $\zeta$- and $\sigma$-function for $E$. Let
$u\in\C$ such that $\wp(u)\in\Qbar$ and $\exp_E(u)$ is non-torsion in $E(\Qbar)$. Then
\[ u\zeta(u)-2\log(\sigma(u))\]
is transcendental.
\end{ex}
In the spirit of Baker's theorem, 
Theorem~\ref{thm:1-mot} also allows us to give an explicit formula for
$\dim_\Qbar\Per\langle M\rangle$ for all $1$-motives $M$. It is not easy to state, so we simply refer to \cite[Chapter~6]{huber-wuestholz}.

On the other hand, Theorem~\ref{thm:1-mot} has consequences for the various conjectures discussed before.
In \cite{ayoub-barbieri}, Ayoub and Barbieri-Viale show that $\onemot_k$ 
is a full subcategory closed under subquotients in $\MMN^\eff(k)$. They also
give descpription of $\Ah(\onemot_k)$ in terms of generators and relations.
Together with Theorem~\ref{thm:1-mot} this implies Kontsevich's period conjecture for curves.

In detail: Following \cite[Definition~4.6]{huber-wuestholz} we call a period number
 of \emph{curve type} if it is a period of a motive of the form
$H^1(C,D)$ with $C$ a smooth affine curve over $\Qbar$ and
$D\subset C$ a finite set of $\Qbar$-valued points.

\begin{cor}[{\cite[Theorem~5.3]{huber-wuestholz}}]
All relations between periods of curve type are induced by bilinearity and functoriality of pairs $(C,D)\to (C',D')$ with $C, C'$ smooth affine curves and
$D\subset C$, $D'\subset C'$ finite sets of points.
\end{cor}

In a different direction:
\begin{cor}[{\cite[Theorem~5.4]{huber-wuestholz}}, \cite{ABB}]
The functor $\onemot_\Qbar\to \VValg$ is fully faithful.
\end{cor}
\begin{proof}
This is Corollary~\ref{cor:full}. An independent proof was given by Andreatta, Barbieri-Viale and Bertapelle.
\end{proof}

On the other hand, Orgogozo \cite{orgogozo} and Barbieri-Viale, Kahn \cite{BVK}
showed that the derived category of $\onemot_k$ is a full subcategory of
$\DMgm(k)$. 

\begin{cor}
The fullness conjecture holds for $d_{\leq 1}\DMgm(\Qbar)$.
\end{cor}
\begin{proof}
By \cite[Proposition~3.2.4]{orgogozo}, the category
$d_{\leq 1}\DMgm(\Qbar)\isom D^b(\onemot_\Qbar)$ has cohomological dimension
$1$. We can now use the same argument as in the proof of Proposition~\ref{prop:full_gm}, unconditionally.
\end{proof}

\begin{cor}[{\cite[Theorem~5.4]{huber-wuestholz}}]
The natural functors $\onemot_\Qbar\to \MMN(\Qbar)$ (non-effective Nori motives) and $\onemot_\Qbar\to\MHS_\Qbar$ (mixed Hodge structures over $\Qbar$)
are full.
\end{cor}
\begin{proof}Their composition with the functor to $\VValg$ is full.
\end{proof}

A much stronger version of the above was recently shown by Andr\'e by comparing the Tannakian dual of the tensor category generated by $1$-motives to the Mumford Tate group.
\begin{thm}[Andr\'e {\cite[Theorem~2.0.1]{andre4}}]
Let $k\subset \C$ be algebraically closed. Let $\onemot_k^\tensor\subset \MMN(k)$ be the full Tannakian subcategory closed under subquotients generated by
$\onemot_k$. 
Then the functor
\[ \onemot_k^\tensor\to \MHS\]
is fully faithful.
\end{thm}
In the case $k=\Qbar$, this means that the $\Qbar$-structure can actually be recovered from the Hodge- and weight filtration -- and conversely. His result also gives fullness of $\onemot_k\to \MMN(k)$ for all algebraically closed $k\subset \C$ (actually he has to show this on the way).

\bibliographystyle{alpha}
\bibliography{periods}

\end{document}